\documentclass[12pt]{amsart}
\usepackage{amssymb}
\usepackage{verbatim}
\usepackage[usenames]{color}
\usepackage{hyperref}
\usepackage{url}
\usepackage[all]{xy}

\newtheorem{theorem}{Theorem}[section]
\newtheorem*{theorem*}{Theorem}
\newtheorem{proposition}[theorem]{Proposition} 
 
\newtheorem{lemma}[theorem]{Lemma}
\newtheorem{cor}[theorem]{Corollary}

\theoremstyle{definition}
\newtheorem{definition}[theorem]{Definition}

\newtheorem{question}[theorem]{Question}

\theoremstyle{remark}

\newcommand{\defemp}{\textit}

\newcommand{\forces}{\Vdash}
\newcommand{\zf}{\textrm{ZF}}
\newcommand{\zfc}{\textrm{ZFC}}
\newcommand{\ch}{\textnormal{CH}}
\newcommand{\ma}{\textnormal{MA}}

\newcommand{\dom}{\textnormal{Dom}}

\newcommand{\cf}{\textnormal{cf}}

\newcommand{\mf}{\mathfrak}
\newcommand{\mc}{\mathcal}
\newcommand{\mbb}{\mathbb}

\newcommand{\p}{\mathcal{P}}

\newcommand{\baire}{{^\omega \omega}}
\newcommand{\bairenodes}{{^{<\omega}\omega}}
\newcommand{\cantor}{{^{\omega}2}}

\newcommand{\exit}{\textnormal{Exit}}

\title{Disjoint Borel Functions}
\author{Dan Hathaway}

\address{Mathematics Department \\
University of Denver\\
Denver, CO 80208, U.S.A.}

\email{Daniel.Hathaway@du.edu}

\thanks{A portion of the results of this paper
 were proven during the September 2012 Fields Institute Workshop
 on Forcing while the author was supported by the Fields Institute.
Work was also done while under NSF grant DMS-0943832.
}

\begin{document}

\begin{abstract}
For each $a \in \baire$, we define
 a Baire class one function
 $f_a : \baire \to \baire$ which
 encodes $a$ in a certain sense.
We show that for each Borel
 $g : \baire \to \baire$,
 $f_a \cap g = \emptyset$ implies
 $a \in \Delta^1_1(c)$
 where $c$ is any code for $g$.
We generalize this theorem
 for $g$ in larger pointclasses $\Gamma$.
Specifically, if $\Gamma = \mathbf{\Delta}^1_2$,
 then $a \in L[c]$.
Also for all $n \in \omega$,
 if $\Gamma = \mathbf{\Delta}^1_{3 + n}$,
 then $a \in \mathcal{M}_{1 + n}(c)$.
\end{abstract}

\maketitle

\section{Introduction}

\begin{definition}
A challenge-response relation
 (\defemp{c.r.-relation})
 is a triple $\langle R_-, R_+, R \rangle$ such that
 $R \subseteq R_- \times R_+$.
The set $R_-$ is the set of \defemp{challenges},
 and $R_+$ is the set of \defemp{responses}.
When $c R r$, we say that $r$ \defemp{meets} $c$.
\end{definition}

\begin{definition}
A backwards generalized Galois-Tukey connection
 (\defemp{morphism}) from
 $\mathcal{A} = \langle A_-, A_+, A \rangle$ to
 $\mathcal{B} = \langle B_-, B_+, B \rangle$ is a pair
 $\langle \phi_-, \phi_+ \rangle$ of functions
 $\phi_- : B_- \to A_-$ and
 $\phi_+ : A_+ \to B_+$
 such that
 $$(\forall c \in B_-)(\forall r \in A_+)\,
 \phi_-(c)\, A\, r \Rightarrow
 c\, B\, \phi_+(r).$$
\end{definition}
When there is a morphism from
 $\mathcal{A}$ to $\mathcal{B}$,
 let us say that
 $\mathcal{A}$ is \textit{above}
 $\mathcal{B}$ and
 $\mathcal{B}$ is \textit{below}
 $\mathcal{A}$.

\begin{definition}
The \defemp{norm} of a c.r.-relation
 $\mathcal{R} =
 \langle R_-, R_+, R \rangle$ is
 $$|| \mathcal{R} || := \min \{
 |S| : S \subseteq R_+ \mbox{ and }
 (\forall c \in R_-)(\exists r \in S)\,
 c\, R\, r \}.$$
\end{definition}

If there is a morphism from $\mc{A}$
 to $\mc{B}$, then
 $||\mc{A}|| \ge ||\mc{B}||$.
Challenge-response relations and morphisms between
 them were introduced by Vojtas as a way to abstract
 features of the study of cardinal charcteristics of the continuum.
For more on c.r.-relations,
 see \cite{Blass} and \cite{Vojtas}.

Temporarily fix a pointclass $\Gamma$.
Let $\mathcal{F}_\Gamma$ be the set of functions from
 $\baire$ to $\baire$ in $\Gamma$.
Let $D$ be the binary relation of disjointness
 of functions from $\baire$ to $\baire$.
That is, given two functions $f, g : \baire \to \baire$,
 let $$f D g :\Leftrightarrow
 f \cap g = \emptyset
 \Leftrightarrow
 (\forall x \in \baire)\, f(x) \not= g(x).$$
Let $\mathcal{D}_\Gamma$ be the c.r.-relation
 $$\mathcal{D}_\Gamma := \langle
 \mathcal{F}_\Gamma,
 \mathcal{F}_\Gamma,
 D \rangle.$$
In this paper we will be interested in the
 c.r.-relation $\mathcal{D}_\Gamma$ for various
 pointclasses $\Gamma$.

For example, we will be interested in computing
 $||\mathcal{D}_{\mathbf{\Delta}^1_1}||$,
 which is the smallest size of a family of Borel
 functions from $\baire$ to $\baire$ such that
 each Borel function from $\baire$ to $\baire$
 is disjoint from some member of the family.
We will show that
 $||\mathcal{D}_{\mathbf{\Delta}^1_1}|| = 2^\omega$
 by showing that $\mathcal{D}_{\mathbf{\Delta}^1_1}$
 is above a c.r.-relation whose norm is $2^\omega$.
Specifically, we will show that
 $\mathcal{D}_{\mathbf{\Delta}^1_1}$ is above
 $\langle \baire, \baire, \le_{\Delta^1_1} \rangle$,
 where $a \le_{\Delta^1_1} b$ iff $a \in \baire$
 is definable by a $\Delta^1_1$ formula using
 $b \in \baire$ as a parameter.
To define the $\phi_-$ part of the morphism,
 for each $a \in \baire$ we will define a Baire
 class one funciton $f_a : \baire \to \baire$
 (and we will have $\phi_-(a) = f_a$).
The $\phi_+$ part of the morphism will simply map
 each function from $\baire$ to $\baire$ in $\Gamma$
 to any code for that function.
The fact that $\langle \phi_-, \phi_+ \rangle$
 is a morphism is the following statement:
 for each $a \in \baire$ and Borel function
 $g : \baire \to \baire$,
 $$f_a \cap g = \emptyset \Rightarrow
 a \le_{\Delta^1_1} \mbox{any code for } g.$$

We will prove that there is a morphism
 from $\mathcal{D}_{\mathbf{\Delta}^1_1}$
 to $\langle \baire, \baire, \le_{\Delta^1_1} \rangle$
 by proving a general theorem
 (Theorem~\ref{maintheorem})
 which provides a sufficient condition for
 when there exists a morphism from an arbitrary
 $\mathcal{D}_\Gamma$ to an arbitrary
 $\langle \baire, \baire, \prec \rangle$,
 where $\prec$ is an ordering on $\baire$.
Just like the case with
 $\mathcal{D}_{\mathbf{\Delta}^1_1}$,
 we will use the functions $f_a$ for the $\phi_-$ map,
 and the $\phi_+$ map will be ``take any code for''.
Thus, if the appropriate relationship holds
 between $\Gamma$ and $\prec$, then we will have that
 for each $a \in \baire$ and each
 $g : \baire \to \baire$ in $\Gamma$,
 $$f_a \cap g = \emptyset \Rightarrow
 a \prec \mbox{any code for } g.$$
We will get that there exists a morphism from
 $\mathcal{D}_{\mathbf{\Delta}^1_2}$ to
 $\langle \baire, \baire, \le_L \rangle$,
 where $a \le_L b$ iff $a \in L[b]$.
The analogous result for larger $\Gamma$
 uses large cardinals.
We will have that as long as
 $\mc{M}_1(b)$
 (the canonical inner model containing
 $1$ Woodin cardinal and containing $b \in \baire$)
 exists for all $b \in \baire$,
 then there is a morphism from
 $\mathcal{D}_{\mathbf{\Delta}^1_3}$
 to $\langle \baire, \baire, \le_{\mc{M}_1} \rangle$,
 where $a \le_{\mc{M}_1} b$ iff
 $a \in \mc{M}_1(b)$.
Next, as long as
 $\mc{M}_2(b)$ exists for all $b \in \baire$,
 there is a morphism from
 $\mathcal{D}_{\mathbf{\Delta}^1_4}$
 to $\langle \baire, \baire, \le_{\mc{M}_2} \rangle$.
The pattern continues like this through
 the projective hierarchy.

In this paper, we are considering
 functions from $\baire$ to $\baire$
 in a pointclass $\Gamma$.
We could have instead considered
 functions in $\Gamma$
 from an arbitrary uncountable Polish space $X$ to
 an arbitrary Polish space $Y$,
 and our results would not change much.
The appropriate encoding function
 $f''_a : X \to Y$
 could be defined by first defining 
 $f'_a : {^\omega 2} \to \baire$
 in
 a way similar to $f_a$
 and then using an injection of
 ${^\omega 2}$ into $X$
 and a surjection of $\baire$ onto $Y$.
We trust that the interested reader can
 work through the details without trouble.

\section{Related Results}

Before considering $\mc{D}_\Gamma$
 for various $\Gamma$,
 we will consider related
 c.r.-relations.
First, consider the everywhere domination
 ordering of functions from $\baire$ to $\omega$.
That is, given $f, g : \baire \to \omega$,
 we write $f \le g$ iff
 $$(\forall x \in \baire)\, f(x) \le g(x).$$
Given any pointclass $\Gamma$,
 let $\mc{E}_\Gamma$ be the c.r.-relation
 whose challenges and responses are $\Gamma$ functions
 from $\baire$ to $\omega$,
 and $g$ meets $f$ iff $f \le g$.

Next, consider the pointwise eventual domination
 ordering of functions from $\baire$ to $\baire$.
That is, given $f, g : \baire \to \baire$,
 we write $f \le^* g$ iff
 $$(\forall x \in \baire)
 \{ n \in \omega : f(x)(n) > g(x)(n) \}
 \mbox{ is finite.}$$
Given any pointclass $\Gamma$,
 let $\mc{R}_\Gamma$ be the c.r.-relation whose
 challenges and responses are $\Gamma$ functions
 from $\baire$ to $\baire$,
 and $g$ meets $f$ iff $f \le^* g$.

It is not difficult to see that for any
 reasonably closed pointclass $\Gamma$,
 there is a morphism from
 $\mc{E}_\Gamma$ to $\mc{R}_\Gamma$ and
 there is a morphism from
 $\mc{R}_\Gamma$ to $\mc{D}_\Gamma$.
The relation $\mc{E}_\Gamma$ for a fixed $\Gamma$
 is relatively high up in the hierarchy of
 c.r.-relations, as we will soon see.

Given a sequence $a \in \baire$,
 let $[[a]] := \{ a \restriction l : l \in \omega \}$.
Given a tree $T \subseteq \bairenodes$, let
 $\mbox{Exit}(T)$ be the (Baire class one) function
 $$\mbox{Exit}(T)(x) := \min \{ l :
 x \restriction l \not\in T \}.$$
The following result shows a way of constructing
 a morphism from $\mc{E}_\Gamma$ to another relation
 in a way which does not depend on $\Gamma$:
\begin{theorem}
\label{powerfulrealnumcode}
Fix $a \in \baire$.
If $M$ is an $\omega$-model $\zf$
 such that some $g : (\baire)^M \to \omega$ in $M$
 satisfies $$(\forall x \in (\baire)^M)\, \exit([[a]])(x) \le g(x),$$
 then $a$ is $\Delta^1_1$ definable in $M$
 using $g$ as a predicate.
\end{theorem}
\begin{proof}
Fix $M$ and $g$ satisfying
 the hypothesis of the theorem.
Let $B \subseteq \bairenodes$ be the set
 $$\{ t \in \bairenodes :
 g(x) \ge |t| \mbox{ for all } x \sqsupseteq t
 \mbox{ in } M \}.$$
Note that $B$ is defined (in $M$) by a $\Pi^1_1$ formula
 that uses $g$ as a predicate.
That is, $B$ is $\Pi^1_1$ in $g$.
We claim there is some $l \in \omega$
 satisfying $(\forall l' \ge l)\,
 a \restriction l' \not\in B$.
If not,
 the poset of elements of $B$
 ordered by extension would be ill-founded,
 and therefore would be ill-founded in $M$,
 so there would exist $x \in (\baire)^M$ satisfying
 $(\exists^\infty l' \in \omega)\, g(x) \ge l'$,
 which is impossible.
Now, fix such an $l$.

We claim that for each $l' \ge l$,
 $a(l')$ is the unique $n$ satisfying
 $(a \restriction l') ^\frown n \not\in B$.
Indeed, since $\exit([[a]]) \le g$,
 for each $l' \ge l$ we have
 $$(\forall n \in \omega)\, a(l') \not= n
 \Rightarrow (a \restriction l') ^\frown n \in B.$$
The other direction is given
 by the property we arranged $l$ to have.
Thus, we have the following definition (in $M$) for $a$:
$$a(l') = \begin{cases}
 a(l') & \mbox{if } l' < l, \\ 
 n     & \mbox{if } l' \ge l \mbox{ and }
         (\forall n' \not= n)
         (\forall x \sqsupseteq (a \restriction l') ^\frown n' \mbox{ in } M)\,
         g(x) \ge l'+1.
 \end{cases}$$
Since $\langle a(l') : l' < l \rangle$
 can be coded by a single number,
 we have a $\Pi^1_1$ definition (in $M$) for $a$
 which uses $g$ as a predicate.
We also have a $\Sigma^1_1$ variant:
$$a(l') = \begin{cases}
 a(l') & \mbox{if } l' < l, \\
 n     & \mbox{if } l' \ge l \mbox{ and }
         (\exists x \sqsupseteq (a \restriction l') ^\frown n \mbox{ in } M)\,
         g(x) < l'+1.
 \end{cases}$$
Thus, $a$ is $\Delta^1_1$ definable in $M$
 using $g$ as a predicate.
\end{proof}
Let use write ``$All$'' to refer to
 the pointclass of all pointsets.
\begin{cor}
\label{everywheredomtype1cor}
There is a morphism from
 $\mc{E}_{All}$ to
 $\langle \baire,
 {^{(\baire)}\omega}, \le_{\Delta^1_1} \rangle$.
\end{cor}
\begin{proof}
Fix $a \in \baire$.
Let $f_a := \mbox{Exit}([[a]])$.
By the above theorem taking $M = V$, if
 $g : \baire \to \omega$ satisfies $f_a \le g$,
 then $a$ is $\Delta^1_1$ definable using $g$
 as a predicate.
\end{proof}
\begin{cor}
\label{boreleveryhweredom}
There is a morphism from
 $\mc{E}_{\mathbf{\Delta}^1_1}$ to
 $\langle \baire, \baire, \le_{\Delta^1_1} \rangle$.
\end{cor}
\begin{proof}
Fix $a \in \baire$.
Let $f_a := \mbox{Exit}([[a]])$.
Let $g: \baire \to \omega$ be Borel
 and let $c$ be a code for $g$.
If we can show that $a$ is in every
 $\omega$-model which contains $c$,
 we will have that $a \le_{\Delta^1_1} c$.
Let $M$ be an arbitrary $\omega$-model
 which contains $c$.
Letting $\tilde{g}$ be the function in $M$
 coded by $c$, we have that
 $\tilde{g} = M \cap g$.
Hence, in $M$ we have $f_a \le \tilde{g}$,
 so the theorem above tells us that
 $a \in M$.
\end{proof}

Corollary~\ref{boreleveryhweredom} will be
 improved by our result that
 there is a morphism from
 $\mc{D}_{\mathbf{\Delta}^1_1}$ to
 $\langle \baire, \baire, \le_{\Delta^1_1} \rangle$.
The generalizations of Corollary~\ref{boreleveryhweredom}
 to larger pointclasses $\Gamma$ are also
 improved by our main result
 (Theorem~\ref{maintheorem}) about morphisms
 from $\mc{D}_\Gamma$ to orderings
 $\langle \baire, \baire, \prec \rangle$.
On the other hand,
 we do not have an analogue of
 Corollary~\ref{everywheredomtype1cor}
 with
 $\mc{D}_{All}$;
 here we see a qualitative difference
 between $\mc{E}_{All}$ and
 $\mc{D}_{All}$.

Another difference between
 $\mc{E}_{All}$ and
 $\mc{D}_{All}$ is the ability to encode
 not just an $a \in \baire$
 but an $A \subseteq \baire$:

\begin{proposition}
\label{horizcodeintroprop}
Fix a set $X$.
Fix $A \subseteq X$.
There exists a function
 $f_A : {^\omega X} \to \omega$ such that
 whenever $M$ is a transitive model of $\zf$
 with $X \in M$ and
 $M$ contains some
 $g : ({^\omega X})^M \to \omega$ satisfying
 $$(\forall x \in ({^\omega X})^M)\,
 f_A(x) \le g(x),$$
 then $A \in M$.
Moreover,
 there is some $t \in {^{<\omega} X}$ satisfying
 $$A = \{ z \in X : g(x) \ge |t|+1
 \mbox{ for all } x \sqsupseteq t ^\frown z
 \mbox{ in } M \}.$$
\end{proposition}
\begin{proof}
It suffices to show the second claim.
Let $f_A : {^\omega X} \to \omega$ be the function
 $$f_A(x) := \begin{cases}
 0 & \mbox{if } (\forall l \in \omega)\,
 x(l) \not\in A, \\
 l+1 & \mbox{if } x(l) \in A \mbox{ and }
     (\forall l' < l)\, x(l') \not\in A.
 \end{cases}$$
Define
 $$B := \{ t \in {^{<\omega}X} :
 g(x) \ge |t| \mbox{ for all }
 x \sqsupseteq t \mbox{ in } M \}.$$
We must find a $t \in {^{<\omega}X}$
 satisfying
 $$A = \{ z \in X :
 t ^\frown z \in B \},$$
 and we will be done.
By the hypothesis on $g$
 and the definition of $f_A$,
 for each $z \in X$,
  $z \in A$ implies
  $\langle z \rangle \in B$.
If conversely for each $z \in X$,
 $\langle z \rangle \in B$ implies
 $z \in A$, then we have
 $$A = \{ z \in X : \langle z \rangle \in B \},$$
 and we are done by defining
 $t := \emptyset$.
If not,
 then fix some
 $x_0 \in X$ satisfying
 $\langle x_0 \rangle \in B$ but $x_0 \not\in A$.

Again by the hypothesis on $g$
 and the definition of $f_A$,
 for each $z \in X$,
 $z \in A$ implies
 $\langle x_0, z \rangle \in B$.
Here it is important that
 $x_0 \not\in A$.
Again, if the converse holds
 that $\langle x_0, z \rangle \in B$
 implies $z \in A$,
 then $$A = \{ z \in X : \langle x_0, z \rangle \in B \},$$
 and we are done by defining $t := \langle x_0 \rangle$.
If not, we may fix
 $x_1 \in X$ satisfying
 $\langle x_0, x_1 \rangle \in B$ but $x_1 \not\in A$.
We may continue like this,
 but we claim that the procedure terminates
 in a finite number of steps.

Assume, towards a contradiction,
 that it does not terminate.
The sequence $$x := \langle x_0, x_1, ... \rangle$$
 we have constructed has all its
 initial segments in $B$.
However, $x$ need not be in $M$.
We handle this situation as follows:
 let $T$ be the set of those elements of $B$
 all of whose initial segments are also in $B$.
The tree $T$ is ill-founded because
 $x$ is a path through it.
Since being ill-founded is absolute,
 $T$ has some path $x'$ in $M$.
We now have
 $(\forall l \in \omega)\, g(x') \ge l$,
 which is impossible.
\end{proof}

We immediately have the following:

\begin{cor}
\label{allfuncdelta11cor}
For each $A \subseteq \baire$,
 there is a function
 $f_A : \baire \to \omega$ such that
 whenever $g : \baire \to \omega$ is any function
 which satisfies $f \le g$,
 then $A$ is $\mathbf{\Delta}^1_1$
 in a predicate for $g$.
Thus, there is a morphism from
 $\mc{E}_{All}$ to
 $\langle \p(\baire), \p(\baire),
 \le_{\mathbf{\Delta}^1_1} \rangle$.
\end{cor}
\begin{proof}
Use the above theorem with
 $X = \baire$ and $M = V$.
\end{proof}

Now, a morphism from
 $\mc{D}_{All}$ to
 $\langle \p(\baire), \p(\baire), \prec \rangle$,
 where $\prec$ is any ordering such that
 $(\forall B \in \p(\baire))\,
 | \{ A : A \prec B \} | \le 2^\omega$,
 will imply that
 $||\mc{D}_{All}|| = 2^{2^\omega}$.
However, it is consistent that
 $||\mc{D}_{All}|| < 2^{2^\omega}$
 so there can be no such morphism.
In fact, it is consistent that
 $||\mc{R}_{All}|| < 2^{2^\omega}$.
This contrasts with the fact that
 $||\mc{E}_{All}|| = 2^{2^\omega}$.

To get a model of
 $||\mc{R}_{All}|| < 2^{2^\omega}$,
 it suffices to get a model in which
 $\mf{b} = \mf{c}$
 (so that there is a scale in
 $\langle \baire, \le^* \rangle$ of length
 $\mf{c}$)
 and the cofinality
 $\cf \langle {^\mf{c} \mf{c}}, \le \rangle$
 of all functions from $\mf{c}$ to $\mf{c}$
 ordered by everywhere domination
 is $< 2^\mf{c}$.
By $\langle {^\lambda \lambda}, \le^* \rangle$
 we mean the set of functioms from $\lambda$
 to $\lambda$ ordered by domination mod
 $< \lambda$.
By $\mf{b}$ we mean the \textit{bounding number},
 and $\mf{c} = 2^\omega$.
To get the required model, we first force so that
 1) $\mf{t} = \mf{c}$
 (where $\mf{t}$ is the \textit{tower number}),
 2) $\mf{c}$ is regular,
 3) $\mf{c}^{<\mf{c}} = \mf{c}$, and
 4) $\mf{c}^+ < 2^\mf{c}$.
Then, we force to add
 $\mf{c}^{++}$ Cohen subsets of $\mf{c}$.
This preserves 1)-4).
Finally, we force by the
 generalization of Hechler forcing
 in \cite{Cummings} to cofinally embed
 $\langle \mf{c}^+, \le \rangle$ into
 the poset of functions from $\mf{c}$ to
 $\mf{c}$ ordered by eventual mod $< \mf{c}$
 domination ($\le^*$).
A simple observation shows that
 $\cf \langle {^\mf{c} \mf{c}}, \le \rangle =
  \cf \langle {^\mf{c} \mf{c}}, \le^* \rangle$,
 and we are done.

For the last result of this section,
 let $\mc{V}_{\mathbf{\Delta}^1_1}$ be the
 c.r.-relation whose challenges and responses
 are Borel functions from
 $\baire \times \baire$ to $\omega$,
 and $g$ meets $f$ iff
 $(\forall x \in \baire)(\exists y \in \baire)\,
 f(x,y) = g(x,y)$.
By Theorem~\ref{maintheorem} we will have that
 $||\mc{D}_{\mathbf{\Delta}^1_1}|| = 2^\omega$.
It is natural to ask whether
 $||\mc{V}_{\mathbf{\Delta}^1_1}|| = 2^\omega$.
The answer is no for the following reason:
 fix an $\alpha < \omega_1$.
Using the fact that there is
 a universal $\mathbf{\Sigma}^0_\alpha$ set,
 we can build a function
 $g_\alpha : \baire \times \baire \to \omega$
 whose graph is $\mathbf{\Sigma}^0_{\alpha+1}$
 such that if $f : \baire \times \baire \to \omega$
 is a function whose graph is
 $\mathbf{\Sigma}^0_\alpha$, then $g_\alpha$ meets $f$.
Hence,
 $||\mc{V}_{\mathbf{\Delta}^1_1}|| = \omega_1$

\section{The Encoding Function}

In this section we will define
 the function $f_a : \baire \to \baire$
 which encodes $a \in \baire$
 to be used in Thorem~\ref{maintheorem}.


\begin{definition}[The Encoding Function $f_a$]
Fix $a \in \baire$.
Pick some $A \subseteq \omega$ such that
 $A =_T a$,
 $A$ is infinite, and
 $A \le_T B$ whenever $B$ is an infinite
 subset of $A$.
Here $\le_T$ means Turing reducible to and
 $=_T$ means Turing equivalent to.
Such a set $A$ is easy to construct.
We actually only need $A$ to be
 $\Delta^1_1$ in every infinite subset
 of itself.
Let $\eta : A \to \omega$ be a function
 such that $(\forall n \in \omega)\,
 \eta^{-1}(n)$ is infinite.
Consider an arbitrary $x = \langle x_0, x_1, ...
 \rangle \in \baire.$
Let $i_0 < i_1 < ...$ be the sequence of
 indices listing which numbers $x_i$ are in $A$.
That is, each $x_{i_k} \in A$, but no other
 $x_i$ is in $A$.
Define
 $$f_a(x) := \langle
 \eta(x_{i_0}), \eta(x_{i_1}), ...
 \rangle$$
If there are only finitely many $x_i$ in $A$,
 define $f_a(x)$ to be anything.
\end{definition}

One can check that the function $f_a$
 is Baire class one
 (the pointwise limit of the sequence of
 continuous functions).
One might wonder if we could define
 $f_a$ differently to be continuous but still
 encode $a$ in the sense that
 given any Borel $g : \baire \to \baire$
 satisfying $f_a \cap g = \emptyset$,
 $a$ is in some countable set associated to $g$.
The answer is no for the reason
 that the cofinality of the poset of all
 continuous functions from $\baire$ to $\omega$
 ordered by everywhere domination is $\mf{d}$,
 the dominating number,
 which can be consistently less than $2^\omega$.

\section{Reachability}

In this section we introduce
 some combinatorial lemmas
 needed for the main theorem.
The results may be of independent
 interest to the reader.

\begin{definition}
Fix $h : \bairenodes \to \omega$,
 $A \subseteq \omega$, and
 $t_1, t_2 \in \bairenodes$.
We write
 $$t_2 \sqsupseteq_h t_1$$
 and say that $t_2$ is an extension of
 $t_1$ \defemp{to the right} of $h$
 iff $t_2 \sqsupseteq t_1$ and
 $(\forall n \in \dom(t_2) - \dom(t_1))\,
 t_2(n) \ge h(t_2 \restriction n)$.
We write
 $$t_2 \sqsupseteq^A t_1$$
 iff $t_2 \sqsupseteq t_1$ and
 $(\forall n \in \dom(t_2) - \dom(t_1))\,
 t_2(n) \not\in A$.
We write
 $$t_2 \sqsupseteq^A_h t_1$$
 iff both
 $t_2 \sqsupseteq_h t_1$ and
 $t_2 \sqsupseteq^A t_1$.
\end{definition}

\begin{definition}
Given $h_1, h_2 : \bairenodes \to \omega$,
 we write $h_1 \le h_2$ iff
 $$(\forall t \in \bairenodes)\,
 h_1(t) \le h_2(t).$$
\end{definition}

The following notion is crucial
 for the ability to find $\sqsupseteq^A$
 extensions of a node $t$ in a set
 $S \subseteq \bairenodes$.
\begin{definition}
Given $t \in \bairenodes$ and
 $S \subseteq \bairenodes$,
\begin{itemize}
\item $t$ is $0$-$S$-\emph{reachable}
 iff $t \in S$;
\item for $\alpha > 0$,
 $t$ is $\alpha$-$S$-\emph{reachable}
 iff $t$ is $\beta$-$S$-reachable for
 some $\beta < \alpha$ or
 $\{ n \in \omega :
 (\exists \beta < \alpha)\,
 t ^\frown n $ is
 $\beta$-$S$-reachable$\}$ is infinite.
\item $t$ is $S$-\emph{reachable} iff
 $t$ is $\alpha$-$S$-reachable for some
 $\alpha$.
\end{itemize}
\end{definition}
A computation shows the following:
\begin{itemize}
\item $t$ is $S$-reachable iff
 $t$ is $\alpha$-$S$-reachable for some
 $\alpha < \omega^{CK}_1(S).$
\item Given $\alpha < \omega^{CK}_1$,
 the set of all $t$ that are
 $\beta$-$S$-reachable for some
 $\beta < \alpha$ is $\Delta^1_1(S)$.
\end{itemize}
\begin{lemma}[Reachability Dichotomy]
\label{reach_di}
Fix $t \in \bairenodes$,
 $S \subseteq \bairenodes$, and
 $A \subseteq \omega$ which is infinite
 and $\Delta^1_1$ in every infinite subset
 of itself.
Assume $A \not\in \Delta^1_1(S)$.
\begin{itemize}
\item If $t$ is not $S$-reachable, then
 $$(\exists h \in \Delta^1_1(S))
 (\forall t' \sqsupseteq_h t)\,
 t' \not\in S.$$
\item If $t$ is $S$-reachable, then
 $$(\forall h)(\exists t' \sqsupseteq^A_h t)\,
 t' \in S.$$
\end{itemize}
\end{lemma}
\begin{proof}
First, consider the case that
 $t$ is not $S$-reachable.
If $\tilde{t}$ is a node which is
 not $S$-reachable, then there must be
 only finitely many $\tilde{t} ^\frown n$ that
 are $S$-reachable.
For each $\tilde{t}$ that is
 not $S$-reachable, define
 $h(\tilde{t})$ to be the smallest $n$ such that
 $(\forall m \ge n)\,
 \tilde{t} ^\frown m$ is not $S$-reachable.
For each $\tilde{t}$ that is
 $S$-reachable, define
 $h(\tilde{t}) = 0$.
A computation shows that
 $h \in \Delta^1_1(S)$.
This function $h$ witnesses that
 $(\forall t' \sqsupseteq_h t)\, t' \not\in S$.

Consider the second case that
 $t$ is $S$-reachable.
Fix $t$,$S$, and $A$ as in
 the statement of the lemma.
Assume that $t$ is $S$-reachable
 and fix $h : \bairenodes \to \omega$.
We must find some $t' \sqsupseteq^A_h t$
 such that $t' \in S$.

Assume that $t$ is not $0$-$S$-reachable,
 otherwise we are already done by
 setting $t' = t$.
Thus, fix the smallest $\alpha > 0$ such that
 $t$ is $\alpha$-$S$-reachable.

By induction, it suffices to find
 some $n \in \omega$ such that
 $n \not\in A$,
 $n \ge h(t)$, and
 $t ^\frown n$ is $\beta$-$S$-reachable
 for some $\beta < \alpha$.
That is, if we keep doing this,
 then we will have a decreasing sequence
 of ordinals $\alpha_0 > \alpha_1 > ...$
 which must eventually reach $0$,
 at which point we will be done.
Let $$B := \{ n \in \omega :
 (\exists \beta < \alpha)\,
 t ^\frown n \mbox{ is }
 \beta\mbox{-}S\mbox{-}\mbox{reachable}\}.$$
$B$ is infinite and $B \in \Delta^1_1(S)$.
If $B - A$ is infinite,
 we can get the desired $n$.
Now, $B - A$ must be infinite
 because otherwise
 $B \cap A =_T B$ and $B \cap A$ is infinite,
 so $$A \le_{\Delta^1_1} B \cap A =_T
 B \le_{\Delta^1_1} S,$$
 which implies $A \le_{\Delta^1_1} S$,
 a contradiction.
\end{proof}

\section{Main Theorem}

We will prove the main theorem by using
 a variant of Hechler forcing.
In fact, we could have used a slight
 variant of Hechler focing where the
 functions in the conditions are required
 to be strictly increasing
 (see \cite{Baumgartner}).
However, we thought the
 Reachability Dichotomy
 (Lemma~\ref{reach_di})
 was worth presenting for its own sake,
 and that lemma
 encapsulates the relevant rank analysis
 corresponding to what was carried out
 in \cite{Baumgartner}.

\begin{definition}
$\mathbb{H}$ is the poset of all pairs
 $(t,h)$ such that $t \in \bairenodes$
 and $h : \bairenodes \to \omega$,
 where $(t_2, h_2) \le (t_1, h_1)$ iff
 $t_2 \sqsupseteq_{h_1} t_1$ and $h_2 \ge h_1$.
Given $A \subseteq \omega$,
 we write $(t_2, h_2) \le^A (t_1, h_1)$
 iff $t_2 \sqsupseteq^A_{h_1} t_1$
 and $h_2 \ge h_1$.
\end{definition}

From the Reachability Dichotomy
 follows the Main Lemma.
Recall that
 $(\forall x,y \in \baire)\,
 x \in \Delta^1_1(y)$ iff every
 $\omega$-model $M$ which contains $y$
 also contains $x$.
\begin{lemma}[Main Lemma]
\label{main_lemma}
Let $M$ be an $\omega$-model of $\zf$
 and $U \in \mc{P}^M(\mbb{H}^M)$ be a set
 dense in $\mbb{H}^M$.
Let $A \subseteq \omega$ be infinite
 and $\Delta^1_1$ in every infinite
 subset of itself but
 $A \not\in M$.
Then
 $$(\forall p \in \mbb{H}^M)
   (\exists p' \le^A p)\, p' \in
   U.$$
\end{lemma}
\begin{proof}
Define
 $$S := \{ t \in \bairenodes :
 (\exists h \in M)\, (t,h) \in U \}.$$
We have $S \in M$.
It must be that
 $A \not\in \Delta^1_1(S)$,
 because otherwise since $M$ is an $\omega$-model,
 we would have $A \in M$.

Now fix an arbitrary $p = (t,h) \in \mbb{H}^M$.
We must find some $p' = (t',h')
 \le^A (t,h)$ such that $p' \in U$
 (and so $h' \in M$).
It suffices to find some $t' \in S$ such that
 $t' \sqsupseteq^A_h t$.

There are two cases: $t$ is $S$-reachable or not.
If $t$ is not $S$-reachable, then by the Reachability
 Dichotomy (Lemma~\ref{reach_di})
 there is $h \in \Delta^1_1(S)$ such that
 $(\forall t' \sqsupseteq_h t)\, t' \not\in S$.
Since $M$ is an $\omega$-model and $S \in M$,
 such an $h$ would be in $M$.
Unpacking the definition of $S$,
 we get that $U$ is not dense in $\mbb{H}^M$, a contradiction.

The other case is that $t$ is $S$-reachable.
Lemma~\ref{reach_di} gives us a $t' \in S$ such that
 $t' \sqsupseteq^A_h t$, which is what we wanted.
\end{proof}

This next theorem refers to the function
 $f_a$ defined in Section 3.

\begin{theorem}[Main Theorem]
\label{maintheorem}
Let $\Gamma$ be the pointclass
 of all sets defined by formulas in
 a certain class (so it makes sense to
 talk about $\Gamma$-formulas).
Let $\prec$ be an ordering on $\baire$
 such that whenever $c,a \in \baire$
 are such that $a \not\prec c$, then
 there exists an $\omega$-model $M$
 of $\zf$ such that
\begin{itemize}
\item $c \in M$;
\item $a \not\in M$;
\item $\mc{P}^M(\mbb{H}^M)$
 is countable (in $V$);
\item for every forcing extension $N$
 (in $V$) of $M$ by $\mbb{H}^M$,
 the truth (in $V$) of $\Gamma$ formulas
  with real parameters in $N$ can be
  computed in $N$.
\end{itemize}

Then for any $a \in \baire$ and
 $g : \baire \to \baire$
 in $\Gamma$,
\begin{eqnarray*}
 f_a \cap g = \emptyset
 & \Rightarrow &
 a \prec (\mbox{any code for }g).
\end{eqnarray*}
\end{theorem}
\begin{proof}
Fix $a$, $g$, and an arbitrary code $c$
 for $g$.
In any model $N$ which contains $c$
 and which can compute the truth (in $V$)
 of $\Gamma$ formulas with real parameters in $N$,
 let $\tilde{g}$ refer to the
 function $g \cap N$ (which is in $N$).
Suppose $a \not\prec c$.
Fix an $\omega$-model $M$ as in the
 hypothesis of the theorem.
Let $A \subseteq \omega$ be the set
 from the definition of $f_a$
 that is $\Delta^1_1$
 in every infinite subset of itself
 and $a =_T A$.
Note that $A \not\in M$.

We will construct an $x \in \baire$ satisfying
 $f_a(x) = g(x)$ and this will prove the theorem.
Let $$\langle U_n \in \p^M(\mbb{H}^M) :
 n < \omega \rangle$$
 be an enumeration (in $V$)
 of the dense subsets of $\mbb{H}^M$ in $M$.
Let $\dot{x}$ be the canonical name
 for the generic real added by
 $\mbb{H}^M$.
We will construct a decreasing
 sequence of conditions of $\mbb{H}^M$
 which hit each $U_n$.
The $x \in \baire$ will be the union of
 the stems in this sequence
 (and it will be generic over $M$
 having the name $\dot{x}$).

Starting with $1 \in \mbb{H}^M$,
 apply the Lemma~\ref{main_lemma} to get $p_0 \le^A 1$
 in $U_0$.
Then, apply Lemma~\ref{main_lemma} again to get
 $p_0' \le^A p_0$ and $m_0 \in \omega$
 such that $(p_0' \forces
 \tilde{g}(\dot{x})(0) = \check{m}_0)^M$.
Next, extend the stem of $p_0'$ by one to get
 $p_0'' \le p_0'$ to ensure that
 $f_a(x)(0) = m_0$.

Next, get $p_1'' \le p_1' \le^A p_1 \le^A p_0''$
 such that $p_1 \in U_1$,
 $(p_1' \forces
 \tilde{g}(\dot{x})(1) = \check{m}_1)^M$
 for some $m_1 \in \omega$, and
 $p_1''$ extends the stem of $p_1'$ by one
 to ensure that $f_a(x)(1) = m_1$.
Continue forever like this.

The $x$ we have constructed is generic for
 $\mbb{H}^M$ over $M$.
Let $N = M[x]$.
For each $n \in \omega$ we have
 $(\tilde{g}(x)(n) = m_n)^N$.
Since $\Gamma$-formulas are absolute between
 $N$ and $V$, for each $n \in \omega$ we have
 $$g(x)(n) = m_n.$$
On the other hand, for each $n \in \omega$
 we have $f_a(x)(n) = m_n$.
\end{proof}

In the following,
 $\mc{M}_n(y)$ refers to the cannonical
 proper class model with $n$ Woodin cardinals
 which contains $y \in \baire$.
For each $n \in \omega$ and $y \in \baire$,
 $\baire \cap \mc{M}_n(y)$ is countable.
When we write $a \in \mc{M}_n(c)$,
 we will be making the assumption that
 $\mc{M}_n(c)$ exists, which has large
 cardinal strength.

\begin{cor}
\label{maincor}
Fix $a \in \baire$,
 $\Gamma$,
 $g : \baire \to \baire$ in $\Gamma$,
 and a code $c$ for $g$.
Assume $f_a \cap g = \emptyset$.
\begin{itemize}
\item $\Gamma = \mathbf{\Delta}^1_1
 \Rightarrow
 a \in \Delta^1_1(c)$;
\item $\Gamma = \mathbf{\Delta}^1_2
 \Rightarrow
 a \in L(c)$;
\item $\Gamma = \mathbf{\Delta}^1_3
 \Rightarrow
 a \in \mc{M}_1(c)$;
\item $\Gamma = \mathbf{\Delta}^1_4
 \Rightarrow
 a \in \mc{M}_2(c)$;
\item ...
\end{itemize}
\end{cor}
\begin{proof}
The first bullet holds because
 $\Delta^1_1$ formulas are absolute
 between $\omega$-models and $V$,
 and whenever $a \not\in \Delta^1_1(r)$,
 there is some $\omega$-model of $\zf$ which
 contains $r$ but not $a$.
The second bullet holds by
 Shoenfield's Absoluteness Theorem.
The last two bullets hold because a forcing
 extension of $\mc{M}_{1 + n}$ below its bottom
 Woodin cardinal can compute the truth of
 $\mathbf{\Delta}_{3 + n}$ formulas with real parameters
 in $N$.
For more information related to the last
 two bullets, see Lemma 4.6 of [Steel].
\end{proof}

From the top bullet of this corollary,
 it follows that there is a morphism
 from $\mc{D}_{\mathbf{\Delta}^1_1}$ to
 $\langle \baire, \baire, \le_{\Delta^1_1} \rangle$.
From the second bullet,
 it follows that there is a morphism
 from $\mc{D}_{\mathbf{\Delta}^1_2}$ to
 $\langle \baire, \baire, \le_L \rangle$, etc.

\section{Necessity of Hypotheses}

Let $\Gamma = \bigcup_{n \in \omega}
 \mathbf{\Delta}^1_n$
 be the pointclass of projective sets.
By Corollary~\ref{maincor},
 if $g : \baire \to \baire$ is a projective function
 and $f_a \cap g = \emptyset$,
 then $a \in \bigcup_{n < \omega} \mc{M}_n(c)$
 where $c$ is any code for $g$.
This implies that
 $||\mc{D}_{\Gamma}|| = 2^\omega$.
It is natural to ask whether
 $||\mc{D}_{\Gamma}|| = 2^\omega$ can be proved in
 $\zfc$ alone
 (the assumption that the
 $\mc{M}_n(c)$ exist goes far beyond $\zfc$).
We can ask the following stronger question:
\begin{question}
Does $\zfc$ prove that for each projective
 $g : \baire \to \baire$ there is a countable set
 $G(g) \subseteq \baire$, and for each
 $a \in \baire$ there is a projective function
 $f_a : \baire \to \baire$ such that
 $(\forall a \in \baire)(\forall g)\,$
 $$f_a \cap g = \emptyset \Rightarrow
 a \in G(g)?$$
\end{question}

We do not know how to answer the above question.
The problem is that the functions $f_a$
 for various $a$ may have nothing to do with one another.
We can, however, answer the following:
\begin{question}
\label{projquestion}
Does $\zfc$ prove that there exist functions
 $f_a$ and countable sets $G(g)$
 as in the above question but with the additional
 requirement that the mapping
 $(a,x) \mapsto f_a(x)$ is projective?
\end{question}

We will now argue that the answer to
 Question~\ref{projquestion} is no.
It suffices to show that $\zfc$ does not prove
 there is a pair of mappings
 $a \mapsto f_a$ and $g \mapsto G(g)$ such that
 $(a,x) \mapsto f_a(x)$ is projective and
 $(\forall a \in \baire)(\forall g)$
 $$(\forall x \in \baire)\,
 f_a(x) \le^* g(x) \Rightarrow
 a \in G(g),$$
 because the pointwise eventual domination relation
 is above the disjointness relation.

Consider a model of the following statements:
\begin{itemize}
\item[1)] There is a projective wellordering of the reals
 of ordertype $2^\omega$;
\item[2)] $\neg \ch$;
\item[3)] $\mf{b} = 2^\omega$.
\end{itemize}
Statement 3) is equivalent to saying that each
 subset of $\baire$ of size $< 2^\omega$
 is $\le^*$-dominated by a single element of $\baire$.
The construction of a model in which
 $\ma + \neg \ch$ holds
 (and therefore $\mf{b} = 2^\omega$)
 and there is a projective wellordering of the reals
 is done in \cite{Harrington}.
Consider a given encoding $a \mapsto f_a$
 such that the map $(a,x) \mapsto f_a(x)$ is projective.
The mapping which takes $a \in \baire$ to
 a code for $f_a$ is projective.
Let $\prec$ be the projective wellordering
 given by 1).
For each $b \in \baire$, we may define
 the function $g_b : \baire \to \baire$ as follows:
 $$g_b(x) := \mbox{ the } {\prec\mbox{-least }} y \in \baire \mbox{ such that }
 (\forall a \prec b)\, f_a(x) \le^* y.$$
Note that the prewellordering $\prec$ is used twice.
Because $\mf{b} = 2^\omega$,
 this function is indeed well-defined.
It is also projective.
Now, consider a set $\mc{A} \subseteq \p(\baire)$
 of size $\omega_1$.
Since $\neg \ch$, we may fix a single $b$ satisfying
 $(\forall a \in \mc{A})\, a \prec b$.
By definition of $g_b$, we have
 $$(\forall a \in \mc{A})(\forall x \in \baire)\, f_a(x) \le^* g_b(x).$$
On the other hand,
 given the countable set $G(g_b) \subseteq \p(\baire)$,
 it cannot be that $\mc{A} \subseteq G(g_B)$.
Hence, the encoding is not as required.

\section{A Forcing Free Proof}

In Corollary~\ref{maincor}
 we showed that if $g : \baire \to \baire$
 is Borel and $c$ is any code for $g$, then
 $$f_a \cap g = \emptyset \Rightarrow
 a \in \Delta^1_1(c),$$
 where $f_a$ is defined in Section 3.
In this section we will present a different
 and forcing free proof that
 $$f_a \cap g = \emptyset \Rightarrow
 a \in \Sigma^2_1(c).$$
To avoid complications,
 we will actually consider functions
 from $\baire$ to $\cantor$.
The function $f_a$ can be modified
 into a function from $\baire$ to $\cantor$
 by simply replacing $\eta : A \to \omega$
 with $\eta : A \to 2$ in the original
 definition of $f_a$.
We will prove the desired result by
 proving the contrapositive.
That is, fix $a \in \baire$,
 Borel $g : \baire \to \cantor$,
 and a code $c \in \baire$ for $g$.
Fix $A \subseteq \omega$ that is Turing equivalent
 to $a$ and $A$ is computable from every infinite
 subset of itself.
Assume that
 $a \not\in \Sigma^2_1(c)$.
We must construct an $x \in \baire$ such that
 $$f_a(x) = g(x).$$
The following game theoretic notion
 is how we will get a forcing free proof:
\begin{definition}
Given a function $j : \baire \to 2$,
 and an $m \in 2$,
 $\mc{G}(j,m)$ is the game where Player
 I plays a pair $(t,h) \in \mbb{H}$ that is $\le$ the current pair and
 Player II plays a pair $(t,h) \in \mbb{H}$ that
 is $\le^A$ the current pair.
After infinitely many moves, let
 $x \in \baire$ be the union of the first elements
 of the pairs played.
Player II wins iff $j(x) = m$.
We say that
 $(t,h)$ \defemp{ensures} that $j(x) = m$ iff
 Player II has a winning strategy for
 $\mc{G}(j,m)$ where the starting position
 is $(t,h)$.
\end{definition}

\begin{lemma}
If for each $i \in \omega$ and
 $(t,h) \in \mbb{H}$
 there exists $m \in 2$
 and $(t',h') \le^A (t,h)$
 which ensures $g(x)(i) = m$,
 then there exists an $x \in \baire$
 such that $f_a(x) = g(x)$.
\end{lemma}
\begin{proof}
Our $x$ will be the union
 of the first elements of the pairs in the sequence
 we will construct.
Start with the condition $(\emptyset,h) \in \mbb{H}$
 where $h$ is arbitrary.
Let $m_0 \in 2$ and
 $(t_0,h_0) \le^A (t,h)$ be such that
 $(t_0,h_0)$ ensures $g(x)(0) = m_0$.
Fix a winning strategy $\eta_0$ for Player II
 for the corresponding game.
Have Player II play according to $\eta_0$
 for one move to get
 $(t'_0, h'_0) \le^A (t_0, h_0)$.
Extend $t'_0$ by one to get
 $(t''_0,h'_0) \le (t'_0,h'_0)$ so that
 $f_a(x)(0) = m_0$.

Let $m_1 \in 2$ and
 $(t_1, h_1) \le^A (t''_0, h'_0)$
 be such that $(t_1, h_1)$ ensures
 $g(x)(1) = m_1$.
Fix a winning strategy $\eta_1$ for
 Player II for the corresponding game.
Have Player II play according to $\eta_0$
 for one more and according to $\eta_1$
 for one more (in the correct games) to get
 $(t'_1, h'_1) \le^A (t_1, h_1)$.
Extend $t'_1$ by one to get
 $(t''_1, h'_1) \le (t'_1, h'_1)$ so that
 $f_a(x)(1) = m_1$.
Continue like this forever.
\end{proof}

Once the next lemma is proved,
 we will be done.

\begin{lemma}
Assuming $a \not\in \Sigma^2_1(c)$,
 for each Borel $j : \baire \to 2$
 and $(t,h) \in \mbb{H}$, there exists
 $m \in 2$ and $(t',h') \le^A (t,h)$
 which ensures $j(x) = m$.
\end{lemma}
\begin{proof}
This can be proved by induction on the
 rank of $j$ within the
 Baire hierarchy.
The base case is when $j$ is continuous,
 and the proof is immediate.
For the induction step,
 assume that
 $\langle j_n : n \in \omega \rangle$
 is a sequence of Borel functions such that
 $$(\forall x \in \baire)\,
 j(x) = \lim_{n \to \infty} j_n(x).$$
Assume that for each $n \in \omega$
 and $(\tilde{t},\tilde{h}) \in \mbb{H}$,
 there exists $m' \in 2$
 and $(\tilde{t}',\tilde{h}')
 \le^A (\tilde{t},\tilde{h})$
 which ensures $j_n(x) = m'$.

Let $n_0 = 0$.
Let $m_0 \in 2$ and
 $(t_0, h_0) \le^A (t,h)$ ensure
 $j_{n_0}(x) = m_0$.
Let $\eta_0$ be a winning strategy
 for Player II for
 $\mc{G}(j_{n_0}, m_0)$.
The strategy $\eta_0$ should be
 applied infinitely often for
 the remainder of the construction
 (assuming it does not terminate).

For $n \in \omega$ and $m \in 2$,
 let $S(n,m) \subseteq \bairenodes$
 be the following set:
 $$S(n,m) := \{ t' \in \bairenodes :
 (\exists n' \ge n)(\exists h')\,
 (t',h') \mbox{ ensures }
 j_{n'}(x) = m \}.$$
There are two cases:
 either $t_0$ is $S(n_0+1,1-m_0)$-reachable
 or not.
First, assume that it is not.
We may fix $\tilde{h} \ge h$
 from Lemma~\ref{reach_di}
 such that
 $(\forall t' \sqsupseteq_{\tilde{h}} t_0)\,
 t' \not\in S(n_0+1,1-m_0)$.
We claim that
 $(t_0, \tilde{h})$ ensures
 $j(x) = m_0$.
To see why, consider the following
 strategy of Player II:
 1) make $\le^A$-extensions to either
 ensure the value of $j_n(x)$ for
 all $n \le 1$ (and these values can
 only be ensured to be $m_0$), and
 2) periodically play according to
 the winning strategies being produced
 from the ensuring process.
When the game finishes,
 calling $x$ the real constructed,
 $j_n(x) = m_0$ for all $n \ge n_0$,
 and so also $j(x) = m_0$.

The other case is that
 $t_0$ is $S(n_0 + 1, 1-m_0)$-reachable.
It is important that $t_0$ can reach
 $S(n_0 + 1, 1-m_0)$ by making a
 $\le^A$-extension, instead of an
 arbitrary $\le$-extension.
The set $S(n_0 + 1, 1-m_0)$
 is $\Sigma^2_1(c)$ (because the
 definition of the set existentially
 quantifies over winning strategies
 for a game of real information).
It cannot be that $A$ is
 $\Sigma^2_1$ in $S(n_0 + 1, 1-m_0)$,
 because if it was then by transitivity
 we would have that
 $a$ is $\Sigma^2_1(c)$.
Since $A$ is not $\Sigma^2_1$
 in $S(n_0 + 1, 1-m_0)$,
 it is also not $\Delta^1_1$ in it,
 so by Lemma~\ref{reach_di} we may fix
 $(t_0',h_0) \le^A (t_0,h_0)$ such that
 $t_0' \in S(n_0 + 1, 1-m_0)$.
At this point, apply the strategy $\eta_0$
 one time to get
 $(t_0'', h_0'') \le^A (t_0', h_0)$.
Since $t_0' \in S(n_0+1,1-m_0)$,
 get $n_1 > n_0$, $m_1 = 1 - m_0$, and
 $(t_1, h_1) \le^A (t_0'', h_0'')$
 that ensures
 $j_{n_1}(x) = m_1$.
Let $\eta_1$ be a winning strategy for
 Player II for $\mc{G}(j_{n_1}, m_1)$.
The strategy $\eta_1$, along with $\eta_0$,
 should be applied infinitely often for
 the remainder of the construction
 (assuming it does not terminate).

There are now two cases:
 either $t_1$ is
 $S(n_1 + 1, 1 - m_1)$-reachable or not.
If not, then we are done by reasoning
 similar to before.
If $t_1$ is $S(n_1 + 1, 1 - m_1)$-reachable,
 then we continue the construction
 and the question becomes whether
 it ever terminates.
Suppose, towards a contradiction,
 that the construction does not terminate.
Let $x \in \baire$ be the sequence
 that has been constructed.
For all $i \in \omega$ we have
 $j_{n_i}(x) = m_i$.
However, the $m_i$'s alternate, so the limit
 $\lim_{n \to \infty} j_n(x)$
 cannot exist, which is a contradiction.
\end{proof}

\section{Acknowledgements}

I would like to thank Andreas Blass for reading
 through the arguments here and making suggestions.
I would also like to thank Trever Wilson for
 explaining how much truth a forcing extension
 of $\mc{M}_n$ can compute.
Finally, I would like to thank the referee
 for finding a significant simplification
 in the main theorem.

\end{document}